\documentclass[10pt,leqno]{amsart}
\usepackage{graphicx}
\usepackage{indentfirst,csquotes}
\usepackage[indentafter]{titlesec}
\titleformat{name=\section}{}{\thetitle.}{0.8em}{\centering\scshape}
\titleformat{name=\subsection}[runin]{}{\thetitle.}{0.5em}{\bfseries}[.]
\titleformat{name=\subsubsection}[runin]{}{\thetitle.}{0.5em}{\itshape}[.]
\titleformat{name=\paragraph,numberless}[runin]{}{}{0em}{}[.]
\titlespacing{\paragraph}{0em}{0em}{0.5em}
\titleformat{name=\subparagraph,numberless}[runin]{}{}{0em}{}[.]
\titlespacing{\subparagraph}{0em}{0em}{0.5em}
\usepackage{lipsum}

\usepackage{amssymb,amsthm,amsmath}
\usepackage{xcolor,paralist,hyperref,titlesec,fancyhdr,etoolbox}
\usepackage[all]{xy}
\theoremstyle{definition}
\newtheorem{definition}{Definition}[]

\newtheorem{example}[definition]{Example}
\theoremstyle{plain}
\newtheorem{theorem}[definition]{Theorem}
\newtheorem{corollary}[definition]{Corollary}

\newtheorem{lemma}[definition]{Lemma}
\newtheorem{proposition}[definition]{Proposition}
\newtheorem{problem}[definition]{Problem}
\renewenvironment{proof}{\noindent\textsc{Proof.}\quad}{\qed}

\hypersetup{ colorlinks=true, linkcolor=black, filecolor=black, urlcolor=black }

\newcommand\ZZ{\mathop{}\!\mathbb{Z}}

\newcommand\CC{\mathop{}\!\mathbb{C}}
\newcommand\RR{\mathop{}\!\mathbb{R}}

\newcommand\GL{\mathop{}\!\mathrm{GL}}
\newcommand\Ort{\mathop{}\!\mathrm{O}}

\newcommand\tr{\mathop{}\!\mathrm{tr}}
\newcommand\mo{\mathop{}\!\mathrm{mod}}
\newcommand{\Bigperp}{\mathop{\vcenter{\hbox{\huge$\perp$}}}}

\begin{document}
\title{Conjugacy of Isometries in Real Orthogonal Groups}
\author[Ziyang ZHU]{Ziyang ZHU}
\date{\today}
\address{School of Mathematical Sciences, Capital Normal University, Beijing 100048, China}
\email{zhuziyang@cnu.edu.cn}
\maketitle

\let\thefootnote\relax

\begin{abstract}
We determine all orthogonal transformations of a quadratic space over reals such that any orthogonal transformation which is conjugate to one of them in the linear group is conjugate in the orthogonal group.
\end{abstract} 

\bigskip

\section{Introduction}\label{s1}
To investigate the topological cobordism groups of high-dimensional knots, Levine developed an algebraic cobordism theory based on matrices \cite{Le69a}. Given the striking similarity between the Alexander polynomial of a knot and the characteristic polynomial of an isometry on a quadratic space, Levine further extended this algebraic theory to the framework of isometric spaces \cite{Le69b}, following Milnor's systematic study of the characteristic polynomials and conjugacy classes of such isometries \cite{Mi69}. This approach eventually led to the determination of the general structure of topological cobordism groups for knots in high odd dimensions \cite[Theorem 24]{Le69b}. Conversely, there has been considerable research, such as \cite{Ta24}, on identifying the invariants of the spaces (often constructed via the local-global principle) from which a given polynomial might originate. Milnor \cite[Theorem 2.1]{Mi69} proved that two isometries of a quadratic space over a local field $k$ (with $\mathrm{char}(k)\neq2$) sharing the same irreducible minimal polynomial are conjugate in $\Ort(V)$, which was conjectured by Levine. It is natural to pose the following more general question.

\begin{problem}
Let $V$ be a non-degenerate quadratic space over a field $k$. Determine all orthogonal transformations $\sigma\in\Ort(V)$ such that any orthogonal transformation in $\Ort(V)$ which is conjugate to $\sigma$ in $\GL(V)$ is conjugate to $\sigma$ in $\Ort(V)$.
\end{problem}

When $k$ is a non-archimedean local field with $\mathrm{char}(k)\neq2$, the answer to the above problem is given by \cite[Theorem 1.2]{XZ26}. For archimedean local fields, the only interesting case is $k=\RR$. This is the case treated in the present paper. Throughout the subsequent discussion, we let $V$ be a non-degenerate real quadratic space.

Let $\sigma\in\Ort(V)$. Its characteristic polynomial $f\in\RR[x]$ is reciprocal in the sense that $f=f^*$, where $f^*(x):=f(0)^{-1}x^{\deg(f)}f(x^{-1})$. Write
\[f=\prod_{p\in I(f)}p^{e_p},\]
where $e_p\geq1$ and
\begin{itemize}
\item $I_0(f):=\{p\in\RR[x]\text{ reciprocal}:\deg(p)=1\text{ and }p|f\}\subseteq\{x+1,x-1\}$,

\item $I_1(f):=\{p\in\RR[x]\text{ irreducible and reciprocal}:\deg(p)=2\text{ and }p|f\}$,

\item $I_2(f):=\{qq^*:q\in\RR[x]\text{ monic irreducible and non-reciprocal, }q|f\}$,

\item $I(f):=I_0(f)\cup I_1(f)\cup I_2(f)$.
\end{itemize}

Regard $V$ as an $\RR[x]$-module $V_{\sigma}$. The primary decomposition theorem of a finitely generated module over a principal ideal domain gives rise to a decomposition of $V_{\sigma}$ as $\RR$-linear subspaces
\[V_{\sigma}=\Bigperp_{p\in I(f)}\Bigperp_{l\geq1}V_{p,\sigma}^{(l)}.\]
Each $V_{p,\sigma}^{(l)}$ appearing in this decomposition is a free $\RR[x]/p^l$-module. Moreover, each part in parentheses constitutes a $p$-primary component, which can be described as
\[\{v\in V:p(\sigma)^lv=0\}.\]
Let $\tau\in\Ort(V)$ be another isometry. If $\sigma$ and $\tau$ are conjugate in $\GL(V)$, then the induced $\RR[x]$-modules $V_{\sigma}$ and $V_{\tau}$ are isomorphic. Therefore, $V_{\tau}$ admits an analogous decomposition as above.

Define
\[I_i^{\text{odd}}:=\{(p,l):p\in I_i(f), l\text{ is odd}, V_{p,\sigma}^{(l)}\neq0\}\]
and
\[I_i^{\text{ev}}:=\{(p,l):p\in I_i(f), l\text{ is even}, V_{p,\sigma}^{(l)}\neq0\}.\]
As we shall prove in \S\ref{s2.2}, the components corresponding to $p\in I_2(f)$ and those indexed by $I_0^{\text{ev}}$ contribute no additional invariants to the passage from $\GL(V)$-conjugacy to $\Ort(V)$-conjugacy. By contrast, the components indexed by $I_1^{\text{ev}}$ behave differently: if $I_1^{\text{ev}}\neq\emptyset$, then there exists $\tau\in\Ort(V)$ which is conjugate to $\sigma$ in $\GL(V)$ but not conjugate to $\sigma$ in $\Ort(V)$. Hence, we may assume $I_1^{\text{ev}}=\emptyset$ and it remains to study the components indexed by $I_0^{\text{odd}}$ and $I_1^{\text{odd}}$.

Let $p$ be a monic irreducible factor of $f$, denote the the multiplicity of the elementary divisor $p^l$ as $\epsilon_{p,\sigma}^{(l)}:=\mathrm{rank}_{\RR[x]/p^l}(V_{p,\sigma}^{(l)})$ (that is, the number of Jordan blocks of $\sigma$ under a suitable basis). Write
\[\epsilon_i:=\sum_{(p,l)\in I_i^{\text{odd}}}\epsilon_{p,\sigma}^{(l)},\quad i=0,1.\]
The main result of this paper is the following theorem.

\begin{theorem}\label{1}
Let $\sigma\in\Ort(V)$, and suppose
\[V_{\sigma}^{\text{odd}}:=\Bigperp_{(p,l)\in I_0^{\text{odd}}\cup I_1^{\text{odd}}}V_{p,\sigma}^{(l)}\]
has signature $(r,s)$, $r+s=n$. Then every $\tau\in\Ort(V)$ which is conjugate to $\sigma$ in $\GL(V)$ is conjugate to $\sigma$ in $\Ort(V)$ if and only if $I_1^{\text{ev}}=\emptyset$ and the equation
\[\frac{n}{2}-\frac{\epsilon_0}{2}-\epsilon_1+\sum_{(p,l)\in I_0^{\text{odd}}}w_{p,l}+2\sum_{(p,l)\in I_1^{\text{odd}}}z_{p,l}=r\]
has exactly one integral solution subject to
\[0\leq w_{p,l}\leq\epsilon_{p,\sigma}^{(l)}~((p,l)\in I_0^{\text{odd}});\quad0\leq z_{p,l}\leq\epsilon_{p,\sigma}^{(l)}~((p,l)\in I_1^{\text{odd}}).\]
\end{theorem}

This gives a complete solution in the archimedean case and shows a sharp contrast with the non-archimedean case \cite[Theorem 1.2]{XZ26}. The proof proceeds by calculating the signatures of the individual subspaces associated with specific Jordan blocks, whose sum yields the signature of the total space. We will also provide an alternative formulation of the main result in Corollary \ref{h}.

{\bf{Acknowledgements}.}
The author thanks Fei XU for bringing this problem to his attention, and Yong HU for his helpful and generous comments.

\section{Decompositions Induced by Orthogonal Transformations}\label{s2}
In this section, in addition to introducing fundamental concepts, we demonstrate how to apply the key techniques from \cite{Mi69}.
\subsection{Quadratic Spaces}\label{s2.1}
We begin by reviewing several key concepts from the theory of quadratic forms.
\begin{definition}[Isometric Spaces]\label{2}
A non-degenerate quadratic space $(V,\langle*,*\rangle)$ over $\RR$ is a finite-dimensional $\RR$-linear space $V$ equipped with a non-degenerate symmetric bilinear form $\langle*,*\rangle$ on $V$ such that $\det(\langle e_i,e_j\rangle)\neq0$ for a basis $\{e_1,\cdots,e_n\}$ of $V$. A morphism between quadratic spaces is an isometry, which is an invertible $\RR$-linear map preserving the bilinear form. In particular, the orthogonal group $\Ort(V)$ (also denoted $\Ort(r,s)$) of $(V,\langle*,*\rangle)$ is
\[\Ort(V):=\{\sigma\in\GL(V):\langle\sigma u,\sigma v\rangle=\langle u,v\rangle\text{, for all }u,v\in V\},\]
elements in it are called orthogonal transformations. An isometric space is a triple $(V,\langle*,*\rangle,\sigma)$, where $\sigma\in\Ort(V)$ is an orthogonal transformation. Two isometric spaces are isomorphic if they are isometric as quadratic spaces and the additional orthogonal transformations are orthogonally conjugate.
\end{definition}

The complete invariant of a non-degenerate $(V,\langle*,*\rangle)$ over $\RR$ is the signature $(r,s)$ of $\langle*,*\rangle$, which means the quadratic form of $\langle*,*\rangle$ is $x_1^2+\cdots+x_r^2-x_{r+1}^2-\cdots-x_{r+s}^2$ under a suitable basis. A non-degenerate $2$-dimensional quadratic space over $\RR$ is called a hyperbolic plane if it has signature $(1,1)$.

If $V_1$, $V_2$ are subspaces of a quadratic space $(V,\langle*,*\rangle)$, we say that $V_1$ and $V_2$ are orthogonal, denote as $V_1\perp V_2$, if $\langle v_1,v_2\rangle=0$ for all $v_1\in V_1$, $v_2\in V_2$.

\begin{definition}[Hermitian Spaces]\label{3}
Let $\iota:\CC\to\CC$ be the conjugation involution. A non-degenerate Hermitian space $(V,h)$ over $(\CC,\iota)$ is a finite-dimensional $\CC$-linear space $V$ equipped with a non-degenerate Hermitian form $h$ with respect to the involution $\iota$ such that $\det h(e_i,e_j)\neq0$ for a basis $\{e_1,\cdots,e_n\}$ of $V$.
\end{definition}
The complete invariant of a non-degenerate $(V,h)$ over $\CC$ is also the signature of $h$.

\subsection{An Orthogonal Decomposition}\label{s2.2}
Let $(V,\langle*,*\rangle)$ be a non-degenerate quadratic space over $\RR$. For an orthogonal transformation $\sigma\in\Ort(V)$, let $f$ be its characteristic polynomial. By \cite[eq. (1)]{XZ26}, $f$ has the form
\[f=\prod_{p\in I(f)}p^{e_p},\quad e_p\geq1.\]

Now regard $V$ as an $\RR[\sigma]$-module $V_{\sigma}$. For $p\in I(f)$, one defines the $p$-primary component of $V$ as
\[V_{p,\sigma}:=\{v\in V:p(\sigma)^lv=0,l\gg0\}.\]
The primary decomposition theorem of a finitely generated module over a principal ideal domain gives rise to a decomposition of $V_{\sigma}$ as $\RR$-linear subspaces
\[V_{\sigma}=\bigoplus_{p\in I(f)}V_{p,\sigma},\]
where each $p$-primary component admits a more refined decomposition
\[V_{p,\sigma}=\bigoplus_{l\geq1}V_{p,\sigma}^{(l)}\]
such that $V_{p,\sigma}^{(l)}$ is a free $\RR[x]/p^l$-module.

By \cite[Lemma 3.1]{Mi69}, each $p$-primary component is orthogonal to other $p'$-primary components unless $p'=p^*$. Furthermore, while \cite[Theorem 3.2]{Mi69} guarantees the orthogonality of the refined decompositions for primary components indexed by $I_0(f)$ and $I_1(f)$, \cite[Theorem 2.5]{XZ26} establishes an orthogonal decomposition for those indexed by $I_2(f)$. Hence,
\[V_{\sigma}=\Bigperp_{p\in I(f)}\Bigperp_{l\geq1}V_{p,\sigma}^{(l)}.\]

Let $\tau\in\Ort(V)$ be another orthogonal transformation that is conjugate to $\sigma$ in $\GL(V)$. By repeating the above procedure, we obtain a decomposition
\[V_{\tau}=\Bigperp_{p\in I(f)}\Bigperp_{l\geq1}V_{p,\tau}^{(l)}.\]

It is clear that $V_{\sigma}$ and $V_{\tau}$ are isomorphic as $\RR[x]$-modules; hence, their corresponding primary components are also $\RR[x]$-isomorphic. Our preliminary conclusion is that, if $\sigma,\tau\in\Ort(V)$ are conjugate in $\GL(V)$, then the corresponding components indexed by $I_2(f)$ and $I_0^{\text{ev}}$ are isomorphic as isometric spaces and therefore contribute no additional conjugacy invariants.

\begin{proposition}[{\cite[Corollary 3.5]{XZ26}}]\label{4}
Let $\sigma,\tau\in\Ort(V)$. Suppose $\sigma$ and $\tau$ are conjugate in $\GL(V)$.
\begin{itemize}
\item If $p\in I_2(f)$ and $p=qq^*$, then there is an isomorphism of isometric spaces
\[(V_{q,\sigma}\oplus V_{q^*,\sigma},\sigma)\overset{\sim}{\to}(V_{q,\tau}\oplus V_{q^*,\tau},\tau).\]

\item If $(p,l)\in I_0^{\text{ev}}$, then there is an isomorphism of isometric spaces
\[(V_{p,\sigma}^{(l)},\sigma)\overset{\sim}{\to}(V_{p,\tau}^{(l)},\tau).\]
\end{itemize}
\end{proposition}

However, the components indexed by $I_1^{\text{ev}}$ carry additional Hermitian invariants and therefore should be treated carefully, see Proposition \ref{n}. Consequently, we may reduce the problem in \S\ref{s1} to the cases corresponding to $I_1^{\text{odd}}$ and $I_0^{\text{odd}}$, which are addressed in \S\ref{s3} and \S\ref{s4}, respectively.

\section{Quadratic Irreducible Factors}\label{s3}
In this section, we focus on the components indexed by $I_1(f)$. For $\sigma\in\Ort(V)$, this part is
\[\Bigperp_{(p,l)\in I_1^{\text{odd}}}V_{p,\sigma}^{(l)}\perp\Bigperp_{(p,l)\in I_1^{\text{ev}}}V_{p,\sigma}^{(l)},\]
The core result of this section is that we only need to consider the subspaces of $V=V_{\sigma}$ indexed by $I_1^{\text{odd}}$. We first examine the case of a single $p\in I_1(f)$, and then observe that the general case follows by the same argument.

\subsection{A Complete Invariant}\label{s3.1}
Let $p=x^2+ax+1\in I_1(f)$, its $p$-primary component can be characterized by a specific method as follows.

\begin{proposition}[{\cite[Theorem 3.3]{Mi69}} and {\cite[Proposition 3.2]{XZ26}}]\label{5}
Let $\sigma\in\Ort(V)$, suppose $p\in I_1(f)$. The set of non-degenerate Hermitian $\CC$-linear spaces (with respect to the involution $x\mapsto x^{-1}$)
\[\{(V_{p,\sigma}^{(l)}/pV_{p,\sigma}^{(l)},h_l):l\geq1\}/\text{isometry}\]
determines the isometric space $(V_{p,\sigma},\sigma|_{V_{p,\sigma}})$ completely, where $V_{p,\sigma}=\Bigperp_{l\geq1}V_{p,\sigma}^{(l)}$.
\end{proposition}

To be more precise, for $l\geq1$, if the rank of each free $\RR[x]/p^l$-module $V_{p,\sigma}^{(l)}$ is $n_l$, then each pair $(V_{p,\sigma}^{(l)}/pV_{p,\sigma}^{(l)},h_l)$ has $\CC$-dimension $n_l$ and corresponds to a triple
\[(V_{p,\sigma}^{(l)},\langle*,*\rangle_l,\sigma|_{V_{p,\sigma}^{(l)}})\]
uniquely, where $V_{p,\sigma}^{(l)}$ has $\RR$-dimension $2ln_l$, $\sigma|_{V_{p,\sigma}^{(l)}}$ has minimal polynomial $p^l$ and
\[\tr_{\CC/\RR}(h_l(\overline{u},\overline{v}))=\big\langle(\sigma|_{V_{p,\sigma}^{(l)}}^{-1}p(\sigma|_{V_{p,\sigma}^{(l)}}))^{l-1}u,v\big\rangle_l\]
for $u,v\in V_{p,\sigma}^{(l)}$.

Furthermore, with respect to a suitable basis, $\sigma|_{V_{p,\sigma}^{(l)}}$ is a direct sum of $n_l$ copies of the Jordan block of size $2l$ with elementary divisor $p^l$. If the signature of $h_l$ is $(r_l,s_l)$, $r_l+s_l=n_l$, then there are $r_l$ positive and $s_l$ negative Jordan blocks contained in the matrix representation of $\sigma|_{V_{p,\sigma}^{(l)}}$.

\begin{proposition}\label{n}
Let $\sigma\in\Ort(V)$. If $I_1^{\text{ev}}\neq\emptyset$, then there exists an isometry $\tau\in\Ort(V)$ which is conjugate to $\sigma$ in $\GL(V)$ but is not conjugate to $\sigma$ in $\Ort(V)$.
\end{proposition}
\begin{proof}
Suppose $(p,l)\in I_1^{\text{ev}}$. The associated Hermitian form $h_l$ on $V_{p,\sigma}^{(l)}/pV_{p,\sigma}^{(l)}$ is classified by its signature, so there exists another Hermitian form $h_l'$ of the same dimension with a different signature. By Proposition \ref{5}, $h_l'$ determines an isometric space with the same structural $\RR[x]$-module as $V_{p,\sigma}^{(l)}$. Since $l$ is even, both underlying real quadratic spaces are hyperbolic of the same dimension, and hence are isometric. Replacing the $p^l$-component of $\sigma$ by this new isometric space and leaving all other components unchanged yields $\tau\in\Ort(V)$, one can verify this $\tau$ is what we need.
\end{proof}

In particular, a necessary condition for every isometry in $\Ort(V)$ that is $\GL(V)$-conjugate to $\sigma$ to be $\Ort(V)$-conjugate to $\sigma$ is
$I_1^{\text{ev}}=\emptyset$.

\subsection{The Range of Signatures}\label{s3.2}
Let $p=x^2+ax+1\in I_1(f)$. We now determine the possible signatures where $\sigma|_{V_{p,\sigma}^{(l)}}$ lies in.

\begin{lemma}\label{6}
An indecomposable component of $\sigma|_{V_{p,\sigma}^{(l)}}$ with minimal polynomial $p^l$ must lie in
\begin{itemize}
\item $\Ort(l,l)$, if $l$ is even.

\item $\Ort(l\pm1,l\mp1)$, if $l$ is odd.
\end{itemize}
\end{lemma}
When $l$ is odd, the signature of the orthogonal group where this indecomposable component lies in is determined by the structure of a $1$-dimensional non-degenerate Hermitian space, by Proposition \ref{5}. It is well known that such a Hermitian space is uniquely determined by the signature $(1,0)$ or $(0,1)$.

\begin{proof}
Note that $|a|<2$. Without loss of generality, assume $V_{p,\sigma}^{(l)}=\RR[x]/p^l$, it has a $\RR$-basis
\[\{s^0,s^1,\cdots,s^{l-1}\}\cup\{xs^0,xs^1,\cdots,xs^{l-1}\}\]
by \cite[Lemma 3.1]{XZ26}, where $s=x+a+x^{-1}$. By the arguments in \cite[Proposition 3.2]{XZ26}, the Gram matrix of the symmetric bilinear form given by Proposition \ref{5} under this basis is
\[G=\left(\begin{array}{ccc|ccc}
&&2&&&-a\\
&\cdots&&&\cdots&\\
2&&&-a&&\\
\hline
&&-a&&&2\\
&\cdots&&&\cdots&\\
-a&&&2&&
\end{array}\right)\]
if the signature of $(V_{p,\sigma}^{(l)}/pV_{p,\sigma}^{(l)},h_l)\cong(\CC,h_l)$ is $(1,0)$; and $-G$ if the signature is $(0,1)$. A direct computation gives the characteristic polynomial of $G$ as
\[\big((x-2+a)(x-2-a)\big)^{\lceil l/2\rceil}\big((x+2-a)(x+2+a)\big)^{\lfloor l/2\rfloor},\]
it has $2\times\lceil l/2\rceil$ positive real roots and $2\times\lfloor l/2\rfloor$ negative real roots; and
\[\big((x+2-a)(x+2+a)\big)^{\lceil l/2\rceil}\big((x-2+a)(x-2-a)\big)^{\lfloor l/2\rfloor}\]
for the case $-G$, which has $2\times\lfloor l/2\rfloor$ positive real roots and $2\times\lceil l/2\rceil$ negative real roots.
\end{proof}

\begin{proposition}\label{7}
Let $p\in I_1(f)$. In the context of Proposition \ref{5}, if $l$ is odd and $h_l$ has signature $(r_l,s_l)$, then $\langle*,*\rangle_l$ has signature
\[(ln_l+r_l-s_l,ln_l+s_l-r_l).\]
\end{proposition}
Obviously, $ln_l+r_l-s_l$ and $ln_l+s_l-r_l$ are even integers.

\begin{proof}
By Lemma \ref{6}, the signature $(\alpha_l,\beta_l)$ of $\langle*,*\rangle_l$ satisfies
\[\left\{\begin{array}{l}
\alpha_l=r_l(l+1)+s_l(l-1)\\
\beta_l=r_l(l-1)+s_l(l+1)\\
\alpha_l+\beta_l=2ln_l\\
r_l+s_l=n_l
\end{array}\right.\]
Hence, $\alpha_l=r_l(l+1)+(n_l-r_l)(l-1)=ln_l-n_l+2r_l=ln_l+r_l-s_l$.
\end{proof}

Recall that $\epsilon_{p,\sigma}^{(l)}$ is the geometric multiplicity of the linear map $\sigma|_{V_{p,\sigma}^{(l)}}$. Indeed, this number coincides with $n_l$.

\begin{corollary}\label{8}
Let $p\in I_1(f)$. The signature $(r_p,s_p)$, $r_p+s_p=n_p$ of $\Bigperp_{l\text{ odd}}V_{p,\sigma}^{(l)}$ varies over
\[r_p,s_p\in\left[\frac{n_p}{2}-\sum_{l\text{ odd}}\epsilon_{p,\sigma}^{(l)},\frac{n_p}{2}+\sum_{l\text{ odd}}\epsilon_{p,\sigma}^{(l)}\right]\cap2\ZZ,\quad\frac{n_p}{2}\equiv\sum_{l\text{ odd}}\epsilon_{p,\sigma}^{(l)}(\mo~2).\]
\end{corollary}
\begin{proof}
By Proposition \ref{7},
\[r_p=\sum_{l\text{ odd}}(ln_l+r_l-s_l)=\frac{n_p}{2}-\sum_{l\text{ odd}}\epsilon_{p,\sigma}^{(l)}+2\sum_{l\text{ odd}}r_l,\]
where $\sum_{l\text{ odd}}r_l$ varies over
\[\left\{0,1,\cdots,\sum_{l\text{ odd}}\epsilon_{p,\sigma}^{(l)}\right\}.\]
Moreover, since $\frac{n_p}{2}=\sum_{l\text{ odd}}ln_l$, we have $\frac{n_p}{2}-\sum_{l\text{ odd}}\epsilon_{p,\sigma}^{(l)}=\sum_{l\text{ odd}}(l-1)n_l\in2\ZZ$.
\end{proof}

\subsection{Criterion for Orthogonal Conjugacy}\label{s3.3}
Let $\sigma,\tau\in\Ort(V)$ be conjugate in $\GL(V)$. From \S\ref{s2.2}, after choosing an $\RR[x]$-module isomorphism $V_{\sigma}\cong V_{\tau}$, we identify the corresponding primary components:
\[V=\Bigperp_{p\in I(f)}\Bigperp_{l\geq1}V_p^{(l)}.\]

For the single case $p\in I_1(f)$, one can prove:
\begin{theorem}\label{9}
Let $\sigma\in\Ort(V)$ and $p\in I_1(f)$. Suppose $\Bigperp_{l\text{ odd}}V_{p,\sigma}^{(l)}$ has signature $(r_p,s_p)$, $r_p+s_p=n_p$. Then, the following statements are equivalent:
\begin{itemize}
\item[(i)] Any $\tau\in\Ort(V)$ such that $\tau|_{V_p}$ is conjugate to $\sigma|_{V_p}$ in $\GL(V_p)$ implies they are conjugate in $\Ort(V_p)$.

\item[(ii)] $V_{p,\sigma}^{(l)}=0$ for all even $l$ and the equation
\[\frac{n_p}{2}+\sum_{l\text{ odd}}\left(2z_{p,l}-\epsilon_{p,\sigma}^{(l)}\right)=r_p\]
has only one solution, where $0\leq z_{p,l}\leq\epsilon_{p,\sigma}^{(l)}$ are integers.
\end{itemize}
\end{theorem}
\begin{proof}
By Proposition \ref{n}, the even-degree components must be absent. Since $\sigma$ and $\tau$ are conjugate in $\GL(V)$, there is an $\RR[x]$-module isomorphism
\[\Bigperp_{l\text{ odd}}V_{p,\sigma}^{(l)}\cong\Bigperp_{l\text{ odd}}V_{p,\tau}^{(l)},\]
and we identify both with $\Bigperp_{l\text{ odd}}V_p^{(l)}$. By Proposition \ref{5}, condition (i) holds if and only if the isometry classes of the quadratic spaces
\[V_p^{(1)},V_p^{(3)},V_p^{(5)},\cdots\]
are uniquely determined. Since a non-degenerate real quadratic space is determined by its signature, this is equivalent to the uniqueness of the corresponding signature assignment. By the preceding computation, such signature assignments are in bijection with the integral solutions of the equation in (ii).
\end{proof}

Note that combining these distinct $p\in I_1(f)$ preserves the conclusion of Theorem \ref{9}. Thus, we immediately obtain the following corollary:

\begin{corollary}\label{10}
Let $\sigma\in\Ort(V)$, suppose $\Bigperp_{(p,l)\in I_1^{\text{odd}}}V_{p,\sigma}^{(l)}$ has signature $(r_1,s_1)$, $r_1+s_1=n_1$. Then, the following statements are equivalent:
\begin{itemize}
\item[(i)] For any $\tau\in\Ort(V)$ such that the restrictions of $\sigma$ and $\tau$ to $\Bigperp_{p\in I_1(f)}\Bigperp_{l\geq1}V_{p,\sigma}^{(l)}$ are similar implies that they are orthogonally conjugate.

\item[(ii)] $I_1^{\text{ev}}=\emptyset$ and the equation
\[\frac{n_1}{2}-\epsilon_1+2\sum_{(p,l)\in I_1^{\text{odd}}}z_{p,l}=r_1\]
has only one solution, where $0\leq z_{p,l}\leq\epsilon_{p,\sigma}^{(l)}$ are integers.
\end{itemize}
\end{corollary}

\section{Factors $x\pm1$}\label{s4}
The approach in this section is similar to that of \S\ref{s3}. Suppose $\sigma\in\Ort(V)$, consider
\[\Bigperp_{(p,l)\in I_0^{\text{odd}}}V_{p,\sigma}^{(l)}\perp\Bigperp_{(p,l)\in I_0^{\text{ev}}}V_{p,\sigma}^{(l)},\]
Since the even part is an orthogonal sum of several copies of hyperbolic planes by Proposition \ref{4}, it suffices to study the subspace of $V=V_{\sigma}$ indexed by $I_0^{\text{odd}}$.
\subsection{A Complete Invariant}\label{s4.1}
Parallel to Proposition \ref{5}, the $(x\pm1)$-primary component can be characterized as follows. Here, we only need to consider the odd-degree part, because the bilinear forms defined on the corresponding quotient spaces that determine the even-degree part are skew-symmetric (see \cite[Theorem 3.4]{Mi69}).

\begin{proposition}[{\cite[Proposition 3.3]{XZ26}} and {\cite[Theorem 3.4]{Mi69}}]\label{11}
Let $\sigma\in\Ort(V)$, suppose $p\in I_0(f)\subseteq\{x+1,x-1\}$. The set of non-degenerate quadratic spaces over $\RR$
\[\{(V_{p,\sigma}^{(l)}/pV_{p,\sigma}^{(l)},b_l):\text{odd }l\geq1\}/\text{isometry}\]
determines the isometric space $(V_{p,\sigma}^{\text{odd}},\sigma|_{V_{p,\sigma}^{\text{odd}}})$ completely, where $V_{p,\sigma}^{\text{odd}}=\Bigperp_{l\text{ odd}}V_{p,\sigma}^{(l)}$.
\end{proposition}

To be more precise, suppose each free $\RR[x]/p^l$-module $V_{p,\sigma}^{(l)}$, $l\geq1$ odd, has rank $n_l$, then each pair $(V_{p,\sigma}^{(l)}/pV_{p,\sigma}^{(l)},b_l)$ has $\RR$-dimension $n_l$ and corresponds to a triple
\[(V_{p,\sigma}^{(l)},\langle*,*\rangle_l,\sigma|_{V_{p,\sigma}^{(l)}})\]
uniquely, where $V_{p,\sigma}^{(l)}$ has $\RR$-dimension $ln_l$, $\sigma|_{V_{p,\sigma}^{(l)}}$ has minimal polynomial $p^l$ and
\[b_l(\overline{u},\overline{v})=\big\langle(\sigma|_{V_{p,\sigma}^{(l)}}-\sigma^{-1}|_{V_{p,\sigma}^{(l)}})^{l-1}u,v\big\rangle_l\]
for $u,v\in V_{p,\sigma}^{(l)}$.

Thus, $\sigma|_{V_{p,\sigma}^{(l)}}$ consists of $n_l$ copies of the Jordan block of size $l$ with elementary divisor $p^l$. If the signature of $b_l$ is $(r_l,s_l)$, $r_l+s_l=n_l$, then there are $r_l$ positive and $s_l$ negative Jordan blocks in $\sigma|_{V_{p,\sigma}^{(l)}}$.

\subsection{Criteria for Signatures}\label{s4.2}
Let $p\in I_0(f)$. Parallel to Lemma \ref{6}, we have:
\begin{lemma}\label{12}
An indecomposable component of $\sigma|_{V_{p,\sigma}^{(l)}}$ with minimal polynomial $p^l$ must lie in
\begin{itemize}
\item $\Ort\left(\frac{l\pm1}{2},\frac{l\mp1}{2}\right)$, if $l\equiv1(\mo~4)$.

\item $\Ort\left(\frac{l\mp1}{2},\frac{l\pm1}{2}\right)$, if $l\equiv3(\mo~4)$.
\end{itemize}
The signature of the orthogonal group where this indecomposable component lies in is also determined by the structure of a $1$-dimensional non-degenerate quadratic space over $\RR$.
\end{lemma}
\begin{proof}
Without loss of generality, assume $V_{p,\sigma}^{(l)}=\RR[x]/p^l$, it has a $\RR$-basis
\[\{(x-x^{-1})^i:0\leq i\leq l-1\}\]
by \cite[Proposition 3.3]{XZ26}. Similarly to Lemma \ref{6}, the Gram matrix of the symmetric bilinear form given by Proposition \ref{11} under this basis is
\[G=\left(\begin{array}{cccccc}
&&&&&1\\
&&&&-1&\\
&&&1&&\\
&&\cdots&&&\\
&-1&&&&\\
1&&&&&
\end{array}\right)\]
if the signature of $(V_{p,\sigma}^{(l)}/pV_{p,\sigma}^{(l)},b_l)\cong(\RR,b_l)$ is $(1,0)$; and $-G$ if the signature is $(0,1)$. If $l=2k+1$, one can compute the characteristic polynomial of $G$ is $(x^2-1)^k(x-(-1)^k)$; and $(x^2-1)^k(x+(-1)^k)$ for $-G$.
\end{proof}

\begin{proposition}\label{13}
Let $p\in I_0(f)$. In the context of Proposition \ref{11}, if $l$ is odd and $b_l$ has signature $(r_l,s_l)$, then $\langle*,*\rangle_l$ has signature
\begin{itemize}
\item $\left(\frac{ln_l+r_l-s_l}{2},\frac{ln_l+s_l-r_l}{2}\right)$, if $l\equiv1(\mo~4)$.

\item $\left(\frac{ln_l+s_l-r_l}{2},\frac{ln_l+r_l-s_l}{2}\right)$, if $l\equiv3(\mo~4)$.
\end{itemize}
\end{proposition}
\begin{proof}
By Lemma \ref{12}, if $l\equiv1(\mo~4)$, the signature $(\alpha_l,\beta_l)$ of $\langle*,*\rangle_l$ satisfies
\[\left\{\begin{array}{l}
\alpha_l=r_l\frac{l+1}{2}+s_l\frac{l-1}{2}\\
\beta_l=r_l\frac{l-1}{2}+s_l\frac{l+1}{2}\\
\alpha_l+\beta_l=ln_l\\
r_l+s_l=n_l
\end{array}\right.\]
Hence, $\alpha_l=\frac{ln_l+r_l-s_l}{2}$, $\beta_l=\frac{ln_l+s_l-r_l}{2}$. The $l\equiv3(\mo~4)$ case is similar.
\end{proof}

For $p\in I_0(f)$ and $\chi\in\{1,3\}$, write
\[V_{p,\sigma}^{[\chi]}:=\Bigperp_{l\equiv\chi(\mo~4)}V_{p,\sigma}^{(l)}.\]
Recall that $\epsilon_{p,\sigma}^{(l)}=n_l$ is the geometric multiplicity of the linear map $\sigma|_{V_{p,\sigma}^{(l)}}$. Similarly to Corollary \ref{8}, we have

\begin{corollary}\label{14}
Let $p\in I_0(f)$ and $\chi\in\{1,3\}$. Then the signature $(r_{\chi},s_{\chi})$, $r_{\chi}+s_{\chi}=n_{\chi}$ of $V_{p,\sigma}^{[\chi]}$ varies over
\[r_{\chi},s_{\chi}\in\left[\frac{1}{2}\left(n_{\chi}-\sum_{l\equiv\chi(\mo~4)}\epsilon_{p,\sigma}^{(l)}\right),\frac{1}{2}\left(n_{\chi}+\sum_{l\equiv\chi(\mo~4)}\epsilon_{p,\sigma}^{(l)}\right)\right]\cap\ZZ.\]
\end{corollary}

Compared to Theorem \ref{9}, although there is now an additional discussion regarding the two congruence classes $[1]$ and $[3]$ modulo $4$, the treatment of the problem remains the same provided that the signatures of $b_l$ are appropriately adjusted. Thus, we immediately obtain

\begin{theorem}\label{15}
Let $\sigma\in\Ort(V)$ and $p\in I_0(f)$. Suppose $\Bigperp_{l\text{ odd}}V_{p,\sigma}^{(l)}$ has signature $(r_p,s_p)$, $r_p+s_p=n_p$. Then, the following statements are equivalent:
\begin{itemize}
\item[(i)] Any $\tau\in\Ort(V)$ such that $\tau|_{V_p}$ is conjugate to $\sigma|_{V_p}$ in $\GL(V_p)$ implies they are conjugate in $\Ort(V_p)$.

\item[(ii)] The equation
\[\frac{n_p}{2}+\sum_{l\text{ odd}}\left(w_{p,l}-\frac{1}{2}\epsilon_{p,\sigma}^{(l)}\right)=r_p\]
has only one solution, where $0\leq w_{p,l}\leq\epsilon_{p,\sigma}^{(l)}$ are integers.
\end{itemize}
\end{theorem}

\begin{corollary}\label{16}
Let $\sigma\in\Ort(V)$, suppose $\Bigperp_{(p,l)\in I_0^{\text{odd}}}V_{p,\sigma}^{(l)}$ has signature $(r_0,s_0)$, $r_0+s_0=n_0$. Then, the following statements are equivalent:
\begin{itemize}
\item[(i)] For any $\tau\in\Ort(V)$ such that the restrictions of $\sigma$ and $\tau$ to $\Bigperp_{(p,l)\in I_0^{\text{odd}}}V_{p,\sigma}^{(l)}$ are similar implies that they are orthogonally conjugate.

\item[(ii)] The equation
\[\frac{n_0-\epsilon_0}{2}+\sum_{(p,l)\in I_0^{\text{odd}}}w_{p,l}=r_0\]
has only one solution, where $0\leq w_{p,l}\leq\epsilon_{p,\sigma}^{(l)}$ are integers.
\end{itemize}
\end{corollary}

\section{The Main Theorem}\label{s5}
The proof of the main theorem is immediately now:

\begin{proof}
(of Theorem \ref{1}) By Corollary \ref{10} and Corollary \ref{16}, the orthogonal conjugacy class is completely determined by these local signatures; conversely every admissible assignment is realizable. Hence the orthogonal conjugacy classes in the similarity class of $\sigma$ are in bijection with the integral solutions of the equation in Theorem \ref{1}.
\end{proof}

We can completely characterize the cases where the equation possesses a unique solution. For $(t,l)\in I_0^{\text{odd}}$, let $b_{t,l}$ denote the symmetric bilinear form associated with $V_{t,\sigma}^{(l)}$. For $(p,l)\in I_1^{\text{odd}}$, let $h_{p,l}$ denote the Hermitian form associated with $V_{p,\sigma}^{(l)}$. Define
\[q:=q_0\perp q_1,\]
where
\[q_0:=\Bigperp_{(t,l)\in I_0^{\text{odd}}}(-1)^{(l-1)/2}b_{t,l}\quad\text{and}\quad q_1:=\Bigperp_{(p,l)\in I_1^{\text{odd}}}\tr_{\CC/\RR}(h_{p,l}).\]

\begin{corollary}\label{h}
The equation in Theorem \ref{1} has exactly one solution if and only if at least one of the following conditions holds:
\begin{itemize}
  \item[(i)] $\#I_0^{\text{odd}}+\#I_1^{\text{odd}}\leq1$.
  \item[(ii)] $q$ is (positive or negative) definite.
  \item[(iii)] $\#I_0^{\text{odd}}=1$, $q_0$ is $1$-dimensional, and either $\#I_1^{\text{odd}}=1$ or $q_1$ is definite.
  \item[(iv)] $\#I_0^{\text{odd}}=1$, $q_0$ is a hyperbolic plane, and either $\#I_1^{\text{odd}}=1$ or $q_1$ is definite.
  \item[(v)] $\#I_0^{\text{odd}}=1$, and if $(r_0,s_0)$ denotes the signature of $q_0$, then one of the following two cases occurs:
  \begin{itemize}
    \item $r_0\geq 2$, $s_0=1$, and $q_1$ is positive definite.
    \item $r_0=1$, $s_0\geq2$, and $q_1$ is negative definite.
  \end{itemize}
\end{itemize}
Consequently, every $\tau\in\Ort(V)$ which is conjugate to $\sigma$ in $\GL(V)$ is conjugate to $\sigma$ in $\Ort(V)$ if and only if $I_1^{\text{ev}}=\emptyset$ and at least one of (i)-(v) holds.
\end{corollary}
\begin{proof}
Write
\[W:=\sum_{(t,l)\in I_0^{\text{odd}}}w_{t,l},\quad\text{and}\quad Z:=\sum_{(p,l)\in I_1^{\text{odd}}}z_{p,l}.\]
By the definition of $q_0$ and $q_1$, their signatures are
\[(W,\epsilon_0-W)\quad\text{and}\quad(2Z,2\epsilon_1-2Z),\]
respectively. Hence the equation in Theorem \ref{1} becomes
\begin{align}\label{e1}W+2Z=r-\frac{n}{2}+\frac{\epsilon_0}{2}+\epsilon_1.\end{align}

We first record two elementary observations. For a fixed $W$, the integers $\{w_{t,l}:(t,l)\in I_0^{\text{odd}}\}$ are uniquely determined if and only if
\begin{align}\label{e2}\#I_0^{\text{odd}}\leq1\quad\text{or}\quad W\in\{0,\epsilon_0\}.\end{align}
Indeed, if $\#I_0^{\text{odd}}\geq2$ and $0<W<\epsilon_0$, one can redistribute one unit between two components while preserving their sum. Similarly, for a fixed $Z$, the integers $\{z_{p,l}:(p,l)\in I_1^{\text{odd}}\}$ are uniquely determined if and only if
\begin{align}\label{e3}\#I_1^{\text{odd}}\leq1\quad\text{or}\quad Z\in\{0,\epsilon_1\}.\end{align}

These observations show immediately that conditions (i) and (ii) imply uniqueness. Conversely, suppose that the equation has exactly one solution and that neither (i) nor (ii) holds. We claim that $\#I_0^{\text{odd}}=1$. Indeed, if $I_0^{\text{odd}}=\emptyset$, then $\#I_1^{\text{odd}}\geq2$, by \eqref{e3} we have $Z=0$ or $Z=\epsilon_1$. Thus $q=q_1$ is definite, contrary to our assumption. Suppose instead that $\#I_0^{\text{odd}}\geq2$, by \eqref{e2} we have $W=0$ or $W=\epsilon_0$. Assume first that $W=\epsilon_0$. Since $q$ is not positive definite, $Z<\epsilon_1$. Moreover, $\epsilon_0\geq2$, so one may decrease the total $W$ by $2$ and increase $Z$ by $1$. This preserves \eqref{e1} and yields another admissible solution, a contradiction. The case $W=0$ is analogous. This proves the claim.

Now, let $(t,l)$ be the unique element of $I_0^{\text{odd}}$, and write $w=w_{t,l}$. Then $\epsilon_0=\epsilon_{t,\sigma}^{(l)}$ and $q_0$ has signature $(w,\epsilon_0-w)$. Since neither (i) nor (ii) holds, we also have $\#I_1^{\text{odd}}\geq1$. The pair $(w,Z)$ is uniquely determined by \eqref{e1} precisely when neither of the moves
\[(w,Z)\longmapsto(w-2,Z+1),\quad(w,Z)\longmapsto(w+2,Z-1)\]
is admissible. Equivalently,
\begin{align}\label{e4}Z=\epsilon_1\text{ or }w\leq1,\quad Z=0\text{ or }w\geq\epsilon_0-1.\end{align}
Together with \eqref{e3}, these conditions completely characterize uniqueness.

\begin{itemize}
\item If $\epsilon_0=1$, both conditions in \eqref{e4} are automatic. Hence uniqueness is equivalent to
\[\#I_1^{\text{odd}}=1\quad\text{or}\quad Z\in\{0,\epsilon_1\},\]
that is, either $\#I_1^{\text{odd}}=1$ or $q_1$ is definite. This is exactly (iii).

\item If $\epsilon_0=2$, then \eqref{e4}, together with the failure of (ii), forces $w=1$. Hence, the signature of $q_0$ is $(1,1)$, this is a hyperbolic plane. Again, by \eqref{e3}, uniqueness is equivalent to $\#I_1^{\text{odd}}=1$ or $q_1$ being definite. This is (iv).

\item Finally, suppose $\epsilon_0\geq3$. If $0<Z<\epsilon_1$, then \eqref{e4} would give simultaneously
\[w\leq1\quad\text{and}\quad w\geq\epsilon_0-1,\]
which is impossible. Thus $Z=0$ or $Z=\epsilon_1$.
\begin{itemize}
\item If $Z=\epsilon_1$, then $q_1$ is positive definite and \eqref{e4} implies $w\geq\epsilon_0-1$. The case $w=\epsilon_0$ would make $q$ positive definite, which is excluded by (ii). Hence $w=\epsilon_0-1$, and therefore the signature of $q_0$ is $(\epsilon_0-1,1)$. Thus $r_0\geq2$, $s_0=1$, and $q_1$ is positive definite.

\item If $Z=0$, then $q_1$ is negative definite and \eqref{e4} implies $w\leq1$. The case $w=0$ would make $q$ negative definite, again falling under (ii). Therefore $w=1$, so the signature of $q_0$ is $(1,\epsilon_0-1)$. Thus $r_0=1$, $s_0\geq2$, and $q_1$ is negative definite.
\end{itemize}
These are precisely the two alternatives in (v).
\end{itemize}
Therefore the equation in Theorem \ref{1} has exactly one integral solution if and only if one of (i)-(v) holds.
\end{proof}

Another advantage of Theorem \ref{1} is that it enables the direct computation of the number of orthogonal conjugacy classes.

\begin{corollary}
We have
\[\#\frac{\{\tau\in\Ort(V):\tau\text{ is similar to }\sigma\}}{\Ort(V)\text{-conjugacy}}=N\cdot\prod_{(p,l)\in I_1^{\text{ev}}}(\epsilon_{p,\sigma}^{(l)}+1),\]
where $N$ is the number of the solutions of the equation in Theorem \ref{1}.
\end{corollary}
\begin{proof}
By the bijection established in the proof of Theorem \ref{1}, the components indexed by $I_0^{\text{odd}}\cup I_1^{\text{odd}}$ contribute exactly $N$ orthogonal conjugacy classes. For $(p,l)\in I_1^{\text{ev}}$. The associated Hermitian form has $\CC$-dimension $\epsilon_{p,\sigma}^{(l)}$ and is determined up to isometry by its signature. Therefore, it has exactly $\epsilon_{p,\sigma}^{(l)}+1$ possible isometry classes. Since $l$ is even, the underlying real quadratic space is hyperbolic for every such choice, so these choices do not affect the quadratic space $V$ and are independent of the choices on the other components. The components indexed by $I_2(f)$ and $I_0^{\text{ev}}$ contribute no additional invariants. Hence, by the orthogonal decomposition and the completeness of the above invariants, we get the counting formula.
\end{proof}

Here we provide a computable example to illustrate our theorem.

\begin{example}
Let $(V,\langle*,*\rangle)$ be a $2$-dimensional non-degenerate quadratic space over $\RR$. If the signature is $(2,0)$ or $(0,2)$, since the Gram-Schmidt process can be applied here, all matrices within the similarity class are mutually orthogonally similar. Consequently, Theorem \ref{1} is trivial in this case. For the signature $(1,1)$, fix a suitable basis, one can assume
\[\Ort(V)=\left\{A\in\GL_2(\RR):A\left(\begin{array}{cc}1&0\\0&-1\\\end{array}\right)A^T=\left(\begin{array}{cc}1&0\\0&-1\\\end{array}\right)\right\}.\]
More precisely, $\Ort(V)$ is equal to
\[\left\{A_{\pm}(t)=\left(\begin{array}{cc}\pm\sqrt{1+t^2}&t\\t&\pm\sqrt{1+t^2}\\\end{array}\right):t\in\RR\right\}\cup\left\{B_{\pm}(t)=\left(\begin{array}{cc}\pm\sqrt{1+t^2}&t\\-t&\mp\sqrt{1+t^2}\\\end{array}\right):t\in\RR\right\}.\]
The left part $A_{\pm}(t)$ has characteristic polynomial $(x-(\pm\sqrt{1+t^2}+t))(x-(\pm\sqrt{1+t^2}-t))$ and canonical form
\[\left(\begin{array}{cc}\pm\sqrt{1+t^2}+t&0\\0&\pm\sqrt{1+t^2}-t\\\end{array}\right)\]
in $\GL(V)$; the right part $B_{\pm}(t)$ has characteristic polynomial $(x+1)(x-1)$ and canonical form
\[\left(\begin{array}{cc}1&0\\0&-1\\\end{array}\right)\]
in $\GL(V)$. We treat the following cases separately.
\begin{itemize}
\item $A_{\pm}(t_1)\sim A_{\pm}(t_2)$ in $\GL(V)$ implies $t_1=\pm t_2$. The only non-trivial case is $t_1=-t_2$, the transition matrix in this case is $B_{\pm}(0)$. This situation corresponds to the factors in $I_2(f)$, of course, we need $t_1\neq0$.

\item $A_{\pm}(t_1)\sim A_{\mp}(t_2)$ in $\GL(V)$ is impossible.

\item Note that $B_+(0)$ is not conjugate to $B_-(0)$ in $\Ort(V)$. Moreover, any $B_+(t)$ is conjugate to $B_+(0)$ by transition matrices $A_{\pm}\left(\frac{1}{\sqrt{2}}\sqrt{\sqrt{t^2+1}-1}\right)$ or $B_+\left(\frac{1}{\sqrt{2}}\sqrt{\sqrt{t^2+1}-1}\right)$ in $\Ort(V)$; similarly, any $B_-(t)$ is also conjugate to $B_-(0)$ in $\Ort(V)$. This situation corresponds to the factors in $I_0(f)$.
\end{itemize}
Hence, a complete set of representatives for the $\Ort(V)$-conjugacy classes is
\[\{A_+(t)(t\geq0),\quad A_-(t)(t\geq0),\quad B_+(0),\quad B_-(0)\}.\]
It is worth comparing this conclusion with \cite[Theorem 2.1]{Mi69}. Note that $B_+(0)$ and $B_-(0)$ are not conjugate in $\Ort(V)$, this can also be derived from Theorem \ref{1} by setting $(r,s)=(1,1)$; $\epsilon_0=2$, $\epsilon_1=0$; $\epsilon_{x\pm1,\sigma}^{(1)}=1$, $\epsilon_{x\pm1,\sigma}^{(\geq3)}=0$; and $\epsilon_{p,\sigma}^{(\geq1)}=0$ for $p\in I_1(f)$. In fact, the equation in Theorem \ref{1} now becomes $w_{x+1,1}+w_{x-1,1}=1$, which has two solutions.
\end{example}

\end{document}